\documentclass[reqno, 11pt]{amsart}
 \usepackage{amsmath}
 \usepackage{amssymb}
 \usepackage{amsthm}
 \usepackage{times}
 \usepackage{latexsym}
 \usepackage[mathscr]{eucal}

 \numberwithin{equation}{section}
 
 \newtheorem{theorem}{Theorem}[section]
 \newtheorem{proposition}[theorem]{Proposition}
 \newtheorem{lemma}[theorem]{Lemma}
 
 \newtheorem{remark}[theorem]{Remark}

 \def\slash#1{\setbox0=\hbox{$#1$}#1\hskip-\wd0\hbox to\wd0{\hss\sl/\/\hss}}

 \title[Timelike conformal fields on closed $3$-manifolds]{Timelike conformal fields on closed $3$-manifolds}
 
 \author[Emmanuel Gnandi, Fortun\'e Massamba ]{Emmanuel Gnandi*,  Fortun\'e Massamba**}

 \dedicatory{In memory of Brice Amour Bivou Kieyela}

 \newcommand{\acr}{\newline\indent}
 
 \address{\llap{*\,}INSA de Toulouse\acr
 	Département de G\'enie Mathématique\acr
 	Université de Toulouse\acr
 	135 avenue de Rangueil\acr
 	31077 Toulouse cedex 4\acr
 	 France} 
 \email{kpanteemmanuel@gmail.com, gnandi@insa-toulouse.fr}
 \thanks{}
 
 \address{\llap{**\,}Discipline of Mathematics\acr 
 	School of Agriculture and Science\acr
 	University of KwaZulu-Natal\acr
 	Private Bag X01, Scottsville 3209\acr 
 	South Africa}  
 \email{massfort@yahoo.fr, Massamba@ukzn.ac.za}  
 \thanks{}  
 
 
 
\begin{document}
 
 \begin{abstract}   
  This paper investigates timelike conformal vector fields on closed Lorentzian $3$-manifolds and shows that, although these fields form a broader class than Killing fields, their behavior in dimension three is nonetheless remarkably rigid. After performing a conformal change of the metric so that the vector field becomes unit and Killing, we analyze the geometry of the flow it generates through the framework of stable Hamiltonian structures and basic cohomology. Our main result proves that any nowhere-vanishing timelike conformal vector field necessarily arises as the Reeb vector field of either a Sasakian structure or a co-K\"ahler structure. In other words, every such Lorentzian conformal flow is intrinsically "Reeb-like", which forces the underlying geometry to be either contact or cosymplectic. This establishes a striking connection between Lorentzian geometry, Sasakian and co-K\"ahler structures, and the topology of flows in dimension~$3$.	 
 	
 	\end{abstract} 

 	\maketitle

 \section{Introduction}
 	
 Timelike conformal vector fields on Lorentzian manifolds have been widely studied for their geometric significance and for their role in relativity theory, where they encode symmetries of the underlying spacetime \cite{Kamishima, romero1995, sanchez1997structure, sanchez2006}. The $3$-dimensional case, in particular, has attracted substantial attention in the literature (see for instance, \cite{aazami, aazami1}). In the Riemannian setting, it is well known that conformal vector fields of unit length that are nowhere vanishing must in fact be Killing \cite{obota}. Such vector fields often give rise to classical geometric structures, most notably stable Hamiltonian structures, which include contact and cosymplectic structures as special cases \cite{cieliebak, rechtman}. The study of these geometries has a long and rich tradition, originating in the foundational work on contact and Sasakian geometry \cite{ blair2010riemannian, Boyer2008, sasaki1962differentiable}, as well as on cosymplectic and co-K\"ahler geometry \cite{cappelletti2013survey, goldberg1969integrability, li2008topology, libermann1962quelques}.
 In Lorentzian geometry, the presence of timelike vector fields introduces features that sharply distinguish this setting from the Riemannian one. The interaction between Lorentzian metrics and almost contact structures has been investigated in a variety of contexts \cite{olea2014, romero1995, sanchez1997structure}. However, the specific problem of determining when a timelike conformal vector field arises as a Reeb vector field has remained open. 
 
 In this paper, we address this question in dimension $3$, where geometric structures are particularly rich and classification problems become more tractable (see \cite{scott1983geometries} for more details and references therein).
 
 $3$-dimensional contact and almost contact geometry has been extensively studied, see for instance \cite{blairgoldberg1967topology, li2008topology, olszak1986normal} and the references therein. In particular, the topology of closed $3$-manifolds admitting Killing vector fields is by now well understood \cite{carriere}.
 
 The main theorem of this paper reveals a close connection between Lorentzian geometry and both co-K\"ahler and Sasakian geometry. More precisely, we show that any timelike conformal vector field which forms a strictly larger class than Killing vector fields must nevertheless be the Reeb vector field of either a Sasakian or a co-K\"ahler structure in dimension $3$.
 
 Our main result can be stated as follows. 
 \begin{theorem}
 	Let $M$ be a closed, oriented, smooth three-manifold equipped with a Lorentzian metric $g$.
 	If $R$ is a nowhere-vanishing timelike vector field that is conformal with respect to $g$, then $R$ is the Reeb vector field of either
 	\begin{enumerate}
 		\item[(i)] a Sasakian structure, or
 		\item[(ii)] a co-K\"ahler structure
 	\end{enumerate}
 	on $M$.
 \end{theorem} 
 Recall that a \emph{Sasakian structure} on a $(2n+1)$-dimensional manifold consists of a contact form $\eta$, its Reeb vector field $R$, an endomorphism $\phi$ of the tangent bundle, and a compatible Riemannian metric $h$, all satisfying suitable integrability conditions \cite{Boyer2008,blair2010riemannian}. In dimension $3$, Sasakian geometry coincides with $K$-contact geometry (see \cite{Boyer2008, rukimbira1995, rukimbira1994}),  and can be regarded as the odd-dimensional analogue of Kähler geometry. A co-K\"ahler structure is a normal cosymplectic structure \cite{cappelletti2013survey, goldberg1969integrability, li2008topology}. In $3$ dimensions, it is equivalent to a cosymplectic structure whose Reeb vector field is Killing, also known as a $K$-cosymplectic structure \cite{bazzoni2015k}.
 
 Our theorem shows that a timelike conformal flow on a closed Lorentzian $3$-manifold is always Reeb-like, and that the underlying geometry is necessarily either contact (Sasakian) or cosymplectic (co-K\"ahler). This dichotomy corresponds to the two possible behaviors of the second fundamental form obtained by contracting the volume form with $R$, and it is reflected in the basic cohomology of the foliation generated by the flow.
 
 Our work connects several areas of geometry, including Lorentzian conformal geometry, foliation theory \cite{epstein1972periodic, prieto2001euler, saralegui1985euler}, stable Hamiltonian structures \cite{ cardona2025, banyaga1990note, acakpo2022, Hutching2009}, and the theory of almost contact structures \cite{bazzoni2015k, cappelletti2013survey}. The $3$-dimensional setting is particularly interesting due to the strong topological constraints \cite{scott1983geometries}, as well as the special properties of contact and cosymplectic structures  \cite{cappelletti2013survey, chinea1993topology, li2008topology, monna1984feuilletages, olszak1986normal, olszak1981}.
 
 The paper is organized as follows. Section~\ref{Pelimina} reviews preliminary material on conformal vector fields, stable Hamiltonian structures, and the basic cohomology of Riemannian foliations. Section~\ref{sec:Result} presents the metric adjustments that yield the stable Hamiltonian structure $(\Omega,\theta)$ and analyzes the basic cohomology of the flow to establish the main theorem. The paper concludes with several corollaries.

 \section{Preliminaries}\label{Pelimina}
 
 We begin this section with definitions and results that will be used throughout the rest of the paper.
 Let $(M,g)$ be a (pseudo-) Riemannian manifold. 
 
 A vector field $X$ on $M$ is called \textit{conformal} if there exists a smooth function $\sigma$ on $M$ such that
 $$
 \mathcal{L}_X g = \sigma g ,
 $$
 where $\mathcal{L}_X g$ denotes the Lie derivative of $g$ along $X$. The function $\sigma$ is referred to as the potential function of $X$. If $\sigma = 0$, then $X$ is a Killing vector field, and its flow consists of isometries of $(M,g)$. More generally, the class of conformal vector fields strictly contains Killing vector fields and, in the special case where $\sigma$ is a nonzero constant, they are called homothetic vector fields.

 The flow generated by a conformal vector field defines conformal transformations, which preserve the metric tensor up to a scalar multiple and thus preserve the conformal structure of the manifold (for more details, see \cite{deshmukh1, deshmukh2}). 
 
 A key fact that will be used repeatedly in this paper is the following classical result.
 
 \begin{lemma}[{\cite[Lemma~2.1]{sanchez1997structure}}]
 	Let $R$ be a nowhere-vanishing timelike conformal vector field on a Lorentzian manifold $(M,g)$. Define a conformally equivalent metric by
 	$$
 	\widetilde{g} = \frac{-1}{g(R,R)}\, g .
 	$$
 	Then $R$ is a Killing vector field with respect to $\widetilde{g}$ and satisfies $\widetilde{g}(R,R) = -1$.
 \end{lemma}
 
 Thus, on a Lorentzian manifold, any timelike conformal vector field can be made unit and Killing by a suitable conformal change of the metric.

 An \emph{almost contact metric structure} on a $(2n+1)$-dimensional manifold $M$ consists of a quadruple $(\eta,\xi,\phi,g)$, where $\eta$ is a $1$-form, $\xi$ is a vector field (called the \emph{Reeb vector field}) satisfying $\eta(\xi)=1$, $\phi$ is an endomorphism of the tangent bundle $TM$ such that
 $$
 \phi^2 = -\mathbb{I} + \eta \otimes \xi ,
 $$
 and $g$ is a Riemannian metric verifying
 $$
 g(\phi X, \phi Y) = g(X,Y) - \eta(X)\eta(Y)
 $$
 for all vector fields $X,Y$ on $M$.
 
 The second fundamental form $\Phi$ associated with the structure is defined by
 $$
 \Phi(X,Y) = g(X,\phi Y).
 $$
 
 The almost contact metric structure is said to be:
 \begin{enumerate}
 	\item \emph{normal} if the  tensor of $[\phi, \phi] + 2d\eta\otimes\xi$ vanishes, where $[\phi, \phi]$ is the Nijenhuis tensor of $\phi$ \cite{blair2010riemannian},
 	\item \emph{$K$-contact} if $\xi$ is Killing and $\eta$ is a contact form, i.e., $\eta \wedge (d\eta)^n \neq 0$, where $\Phi = d\eta$, see for instance \cite{banyaga1995, monna1984feuilletages, yamazaki},
 	\item \emph{Sasakian} if it is both normal and contact \cite{Boyer2008};
 	\item \emph{cosymplectic} if $d\eta = 0$ and $d\Phi = 0$ \cite{li2008topology, libermann1962quelques, cappelletti2013survey},
 	\item \emph{co-K\"ahler} if it is both normal and cosymplectic \cite{cappelletti2013survey, li2008topology}.
 \end{enumerate} 
 A stable Hamiltonian structure elegantly bridges two fundamental geometric worlds, serving as a natural generalization of both cosymplectic and contact structures.
 
 It was first introduced by Hofer and Zehnder as a key condition for hypersurfaces in the context of the Weinstein conjecture \cite{hofer1995}. Later, these structures gained broader importance by providing the precise geometric setting for establishing compactness in symplectic field theory \cite{bourgeois2003, cieliebak}.
 
 A major breakthrough occurred in 2007 when the Weinstein conjecture was proven in full generality for closed $3$-dimensional contact manifolds \cite{taubes2007}. Given that stable Hamiltonian structures generalize contact geometry, a natural next step is to extend this result to their setting. This question remains open today, though the conjecture itself has been adapted to the stable Hamiltonian framework, and several promising advances have been made in \cite{Hutching2009}.
 
 In recent years, research has increasingly turned toward the rich topology of these structures. A growing body of contemporary work explores their classification and applications across symplectic and contact topology \cite{acakpo2022, bourgeois2003, cieliebak,  eliashberg2000,  rechtman}.

 A \textit{stable Hamiltonian structure (SHS)}  on a closed, oriented $3$-manifold $M$ is a pair $(\Omega,\theta)$ consisting of a $1$-form $\theta$ and a $2$-form $\Omega$ satisfying the following conditions: 
 	\begin{align}
 		& d\Omega = 0, \nonumber\\
 		& \theta \wedge \Omega > 0, \nonumber\\
 		& \ker(\Omega) \subset \ker(d\theta).\nonumber
 	\end{align} 
 	The $2$-form $\Omega$ is called the Hamiltonian form, and $\theta$ is said to stabilize $\Omega$.
 
 The positivity condition $\theta \wedge \Omega > 0$ ensures that $\Omega$ is nowhere vanishing. Consequently, the inclusion $\ker(\Omega) \subset \ker(d\theta)$ is equivalent to the existence of a smooth function $\tau : M \to \mathbb{R}$ such that
 $$
 d\theta = \tau \,\Omega.
 $$
 
 The Reeb vector field $R$ associated with $(\theta,\Omega)$ is uniquely defined by the relations
 $$
 \Omega(R,\cdot)=0, \qquad \theta(R)=1.
 $$
 A direct consequence of the definition is the invariance of the structure under the Reeb flow:
 \begin{equation} \label{eq:invariance}
 	\mathcal{L}_{R}\theta = 0, \qquad 
 	\mathcal{L}_{R}\theta = 0,\qquad 
 	\mathcal{L}_{R}(\theta\wedge \Omega) = 0 ,\qquad 
 	\mathcal{L}_{R}\tau = 0 .
 \end{equation}
 
 The kernel of $\Omega$ determines a $1$-dimensional foliation $\mathcal{F}_\Omega$, called the \textit{stable Hamiltonian foliation}. The condition $\theta\wedge\Omega>0$ implies that $\theta$ does not vanish along $\mathcal{F}_\Omega$, that the hyperplane field $\mathcal{F}_\theta = \ker \theta$ is everywhere transverse to $\mathcal{F}_\Omega$, and that $\Omega$ is nondegenerate on $\mathcal{F}_\theta$. Thus, an SHS induces a canonical splitting
 $$
 TM = \mathcal{F}_\Omega \oplus \mathcal{F}_\theta,
 $$
 where $\mathcal{F}_\Omega$ is an oriented line bundle and $\mathcal{F}_\theta$ is a symplectic $2$-plane bundle.
 
 A Darboux theorem holds for stable Hamiltonian structures: around any point of $M$, there exist local coordinates $(q,p,z)$ (called Darboux coordinates) such that \cite[Prop.~10.1]{deleon2024}  
 $$
 \Omega = dq \wedge dp, \qquad \theta = a\,dq + b\,dp + dz.
 $$
 In these coordinates, the condition $\ker\Omega \subset \ker d\theta$ translates into
 $$
 \frac{\partial a}{\partial z} = \frac{\partial b}{\partial z} = 0.
 $$ 
 Stable Hamiltonian structures generalize both contact structures (when $\Omega = d\theta$) and cosymplectic structures (when $d\theta = 0$ and $d\Omega=0$). They play a fundamental role in Hamiltonian dynamics and in the study of Reeb flows, see, for instance, \cite{acakpo2022, cardona2025,  Hutching2009}. 
 
 Note that there exists a Darboux theorem for the cosymplectic case \cite{{Cantrijn1992},  deleon1989} as well as for the contact case \cite{banyaga2017, deleon1989}.

 Let $\mathcal{F}$ be a foliation on $M$. A differential form $\omega$ is said to be \emph{basic} if
 $$
 \iota_X \omega = 0 \quad \text{and} \quad \mathcal{L}_X \omega = 0
 $$
 for every vector field $X$ tangent to $\mathcal{F}$. The space of basic forms is closed under exterior differentiation and defines the basic cohomology $H_B^*(\mathcal{F})$ (see \cite{molino1988, tondeur1997geometry}).
 
 If $\mathcal{F}$ is generated by a non-singular Killing vector field on a Riemannian manifold, then $\mathcal{F}$ is a \emph{Riemannian foliation}, and its basic cohomology groups are finite-dimensional \cite{kacimi1985cohomologie}. For a Riemannian flow on a closed oriented $3$-manifold, the following result holds.
 
 \begin{theorem}\cite[Theorem~A]{molino1985deux},  \cite[Theorem.~10.17]{tondeur1997geometry}\label{thm:mainthe} 
 	Let $\mathcal{F}$ be a Riemannian flow on a closed oriented $3$-manifold. Then
 	$$
 	\dim H^2_B(\mathcal{F}) \leq 1.
 	$$
 \end{theorem}
 
 This result will play a crucial role in the analysis of the differential of the $1$-form naturally associated with our vector field.

 \section{Main Results} \label{sec:Result}

 In the rest of the paper, we adopt the conventions and definitions of co-K\"ahler and Sasakian structures as presented in
 \cite{bazzoni2015k, cappelletti2013survey, li2008topology, libermann1962quelques}.

 \begin{theorem}
 	Let $M $ be a closed, oriented, smooth 3-manifold equipped with a Lorentzian metric $g $.  
 	If $R $ is a nowhere-vanishing timelike vector field that is conformal with respect to $g $, then $R $ is the Reeb vector field of either 
 	\begin{enumerate}
 		\item[(i)] a Sasakian structure, or 
 		\item[(ii)] a co-K\"ahler structure
 	\end{enumerate}
 	on $M $.
 \end{theorem}  
 \begin{proof}
 	Assume that $R$ is timelike, i.e., $g(R,R) < 0$ everywhere. Define the metric
 	$$
 	\widetilde{g}_L := \frac{-1}{g(R,R)}\, g .
 	$$  
 	Then $\widetilde{g}_L(R,R) = -1$. Since $R$ is conformal for $g$, it is also conformal for $\widetilde{g}_L$. A direct computation using \cite[Lemma~2.1]{sanchez1997structure} shows that a conformal vector field must be Killing. Hence 
 	$$
 	\mathcal{L}_R \widetilde{g}_L = 0, \quad \widetilde{g}_L(R,R) = -1,
 	$$ 
 	where $\mathcal{L}_R$ denotes the Lie derivative along $R$.
 	
 	Let $\alpha_L := \widetilde{g}_L(R,\cdot)$. Define a Riemannian metric
 	$$
 	g_R := \widetilde{g}_L + 2\, \alpha_L \otimes \alpha_L .
 	$$  
 	For any vector $X$ orthogonal to $R$ with respect to $\widetilde{g}_L$, we have $g_R(X,X) = \widetilde{g}_L(X,X) > 0$.  
 	For $R$ itself:
 	$$
 	g_R(R,R) = \widetilde{g}_L(R,R) + 2 \bigl(\alpha_L(R)\bigr)^2 = -1 + 2(-1)^2 = 1.
 	$$  
 	Thus $g_R$ is positive definite. Moreover, since $\mathcal{L}_R \widetilde{g}_L = 0$ and $\mathcal{L}_R \alpha_L = 0$, it follows that $\mathcal{L}_R g_R = 0$, i.e., $R$ is Killing also for the Riemannian metric $g_R$. By a standard result (Wadsley \cite[Lemma~3.1]{wadsley1975geodesic}), there exists a Riemannian metric $\widehat{g}$ such that
 	$$
 	\widehat{g}(R,R) = 1, \quad \mathcal{L}_R \widehat{g} = 0, \quad \widehat{\nabla}_R R = 0,
 	$$  
 	where $\widehat{\nabla}$ denotes the Levi-Civita connection of $\widehat{g}$. 
 	(One can take $\widehat{g} = g_R$ and then rescale in directions orthogonal to $R$ if needed; Wadsley's lemma guarantees the geodesic property while preserving the Killing property and unit length.)
 	
 	Define the smooth 1-form
 	$$
 	\theta := \iota_R \widehat{g}.
 	$$  
 	Then $\theta(R) = 1$. Using $\widehat{\nabla}_R R = 0$ and the Killing equation
 	$$
 	\widehat{g}(\widehat{\nabla}_X R, Y) + \widehat{g}(X, \widehat{\nabla}_Y R) = 0,
 	$$
 	we obtain, for any $Y$ 
 	\begin{align}
 		2\, d\theta(R,Y) 
 		&= R\bigl(\theta(Y)\bigr) - Y\bigl(\theta(R)\bigr) - \theta([R,Y]) \nonumber\\
 		&= R\bigl(\widehat{g}(R,Y)\bigr) - Y\bigl(\widehat{g}(R,R)\bigr) - \widehat{g}(R,[R,Y]) \nonumber\\
 		&= \widehat{g}(\widehat{\nabla}_R R, Y) + \widehat{g}(R, \widehat{\nabla}_R Y) - \widehat{g}(R, \widehat{\nabla}_R Y - \widehat{\nabla}_Y R) \nonumber\\
 		&= \widehat{g}(\widehat{\nabla}_R R, Y) + \widehat{g}(R, \widehat{\nabla}_Y R) \nonumber\\
 		&= 0. \nonumber
 	\end{align}  	
 	Let $\mathrm{vol}_{\widehat{g}}$ denote the Riemannian volume form of $\widehat{g}$. Define the smooth $2$-form
 	$$
 	\Omega := \iota_R \mathrm{vol}_{\widehat{g}} .
 	$$  
 	Since $R$ is Killing, $\mathcal{L}_R \mathrm{vol}_{\widehat{g}} = 0$, and thus
 	$$
 	d\Omega = \mathcal{L}_R \mathrm{vol}_{\widehat{g}} - \iota_R d \mathrm{vol}_{\widehat{g}} = 0.
 	$$  
 	A direct check shows
 	$$
 	\iota_R (\theta \wedge \Omega) = \mathrm{vol}_{\widehat{g}} - \theta \wedge \Omega = 0,
 	$$  
 	and since $\theta \wedge \Omega$ is a nowhere-vanishing $3$-form, we have
 	$$
 	\mathrm{vol}_{\widehat{g}} = \theta \wedge \Omega.
 	$$  
 	Hence, $(\Omega, \theta)$ defines a stable Hamiltonian structure \cite{acakpo2022, cardona2025, Hutching2009}.
 	
 	The vector field $R$ generates a 1-dimensional foliation $\mathcal{F}$. Since $R$ is Killing for $\widehat{g}$, $\mathcal{F}$ is a Riemannian foliation. For such a foliation on a closed manifold, the basic cohomology groups are finite-dimensional (see \cite{kacimi1985cohomologie}). Since $\Omega$ is basic ($\iota_R \Omega = 0$ and $d\Omega = 0$), its basic cohomology class $[\Omega]_B$ is well-defined in $H^2_B(\mathcal{F})$.
 	
 	If $\Omega = d \beta$ for a basic 1-form $\beta$, then
 	$$
 	\mathrm{vol}_{\widehat{g}} = \theta \wedge d\beta.
 	$$  
 	But $\theta \wedge d\beta = d(\theta \wedge \beta)$ because $d\theta$ is also basic and $\beta$ is basic. This would imply that $\mathrm{vol}_{\widehat{g}}$ is exact, which leads to a contradiction by Stokes' theorem. Hence $[\Omega]_B \neq 0$ in $H^2_B(\mathcal{F})$ \cite{acakpo2022}. 
 	
 	For a Killing foliation on a closed 3-manifold, the basic cohomology in degree 2 is isomorphic to $\mathbb{R}$ (see Molino \cite[Theorem~A]{molino1985deux}). In fact, $H^2_B(\mathcal{F})$ is generated by $[\Omega]_B$. Therefore, every basic 2-form is a constant multiple of $\Omega$ modulo an exact basic form.
 	
 	Thus, the $2$-form $d\theta$ satisfies $\iota_R d\theta = 0$ and is closed, hence it is basic. So, there exists a constant $k \in \mathbb{R}$ and a basic 1-form $\alpha$ such that
 	$$
 	d\theta = k\, \Omega + d\alpha.
 	$$
 	
 	\paragraph{Case 1: $k = 0$.}  
 	Then $d\theta = d\alpha$. Define
 	$$
 	\widetilde{\theta} := \theta - \alpha.
 	$$  
 	We have $d\widetilde{\theta} = 0$, and since $\alpha$ is basic, $\iota_R \alpha = 0$, hence $\widetilde{\theta}(R) = 1$. Moreover,
 	$$
 	\widetilde{\theta} \wedge \Omega = \theta \wedge \Omega = \mathrm{vol}_{\widehat{g}}.
 	$$  
 	Thus $(\Omega, \widetilde{\theta})$ is a cosymplectic structure with Reeb vector field $R$. Since $R$ is Killing for $\widehat{g}$, the cosymplectic structure is $K$-cosymplectic \cite{bazzoni2015k}. In dimension 3, a $K$-cosymplectic structure is automatically co-K\"ahler \cite{goldberg1969integrability}. Hence, $R$ is the Reeb vector field of a co-K\"ahler structure.\\
 	
 	\paragraph{Case 2: $k \neq 0$.}  
 	Without loss of generality, we may assume $k = 1$. Then
 	$$
 	d\theta = \Omega + d\alpha.
 	$$  
 	Define the smooth 1-form
 	$$
 	\eta := \theta - \alpha.
 	$$  
 	Then $d\eta = \Omega$, so $d\eta$ is non-degenerate on $\ker\eta$, $\eta(R) = 1$, and
 	$$
 	\eta \wedge d\eta = \eta \wedge \Omega = \theta \wedge \Omega = \mathrm{vol}_{\widehat{g}}.
 	$$  
 	Hence $\eta$ is a contact form with dual vector the Reeb vector field $R$. Since $R$ is Killing for $\widehat{g}$, the contact structure is $K$-contact. In dimension 3, every $K$-contact structure is Sasakian \cite{Boyer2008}. Thus $R$ is the Reeb vector field of a Sasakian structure. This concludes the proof.
 \end{proof} 
 \begin{remark}
 \emph{	From the above Theorem \ref{thm:mainthe}, we deduce that, in terms of the bundle isomorphism
 	$$
 	\chi_{\lambda,\Omega} : TM \longrightarrow T^*M, \;\; 
 	v_q \longmapsto \iota_{v_q}\Omega + \lambda(v_q)\,\lambda,
 	$$
 	any nowhere-vanishing timelike vector field $R$ can be recovered as
 	$$
 	R = \chi_{\lambda,\Omega}^{-1}(\lambda),
 	$$
 	where either $ d\lambda = \Omega $ or $d\lambda = 0$. Moreover, in canonical coordinates $(t,p,q)$ adapted to the Reeb flow, the vector field $R$ takes the simple form
 	$$
 	R = \frac{\partial}{\partial t}.
 	$$
 }
 \end{remark}
 
 The following results are direct consequences of the  Theorem \ref{thm:mainthe}. But, first of all, we introduce the concept of K\"ahler mapping torus as follows. If $M$ is a K\"ahler mapping torus, then $M \cong \Sigma_\varphi$, where   $\Sigma_\varphi$ a mapping torus of $\varphi$ defined as
 $$
 \Sigma_{\varphi} = \frac{\Sigma \times [0,1]}{(x,0) \sim (\varphi(x),1)}.
 $$
 and $(\Sigma, J, h)$ is a compact K\"ahler surface and $\varphi$ is a Hermitian isometry of $(\Sigma, J, h)$ (see Li in \cite{li2008topology} for more details).
 
 \begin{proposition}
 	Let $R$ be a nowhere-vanishing timelike conformal vector field on a closed, oriented, smooth $3$-manifold $M$. 
 	If the flow of $R$ has no periodic orbits, then $R$ is the Reeb vector field of a co-K\"ahler structure, and $M$ is a K\"ahler mapping torus. More precisely, $M$ is a $\mathbb{T}^{2}$-bundle over $\mathbb{S}^{1}$.
 \end{proposition}
 
 \begin{proof}
 	By Theorem~\ref{thm:mainthe}, the vector field $R$ is either the Reeb vector field of a Sasakian structure or of a co-K\"ahler structure.
 	
 	It is known that the Reeb vector field of a contact structure admits at least one periodic orbit (see \cite{taubes2007}). Moreover, in the Sasakian case, the Reeb vector field admits at least two closed orbits (see \cite{banyaga1990note, rukimbira1995}). Since, by assumption, the flow of $R$ has no periodic orbits, $R$ cannot be the Reeb vector field of a Sasakian structure. Therefore, $R$ must be the Reeb vector field of a co-K\"ahler structure. By Li \cite[Theorem~B]{li2008topology}, $M$ is a K\"ahler mapping torus ($M$ fibres over $\mathbb{S}^{1}$, see \cite{tischler1970}). Hence, $M$ is either an $\mathbb{S}^{2}$-bundle over $\mathbb{S}^1$, a $\mathbb{T}^{2}$-bundle over $\mathbb{S}^1$, or a $\Sigma_{s}$-bundle over $\mathbb{S}^1$ with genus $s>1$. 
 	
 	Recall that a co-K\"ahler structure is a special case of a stable Hamiltonian structure. Then, by \cite[Theorem 1.1]{Hutching2009}, in the cases of an $\mathbb{S}^{2}$-bundle or a $\Sigma_{s}$-bundle with $s>1$, the Reeb vector field must admit periodic orbits. Since the flow of $R$ has no periodic orbits, these possibilities are excluded. Therefore, $M$ must be a $\mathbb{T}^{2}$-bundle over $\mathbb{S}^1$, as claimed.
 \end{proof} 
 \begin{remark}
 \emph{	Note that if the flow of $R$ has dense orbits, then by Carri\`ere \cite{carriere}, the manifold $M$ is diffeomorphic to the $3$-torus $\mathbb{T}^3$.}
 \end{remark}
 
 \begin{proposition}\label{thm:remak2}
 	Let $R$ be a nowhere-vanishing timelike conformal vector field on a closed, oriented, smooth $3$-manifold $M$.
 	If the first Betti number of $M$ is even, then $R$ is the Reeb vector field of a Sasakian structure.
 	Moreover, $R$ admits at least two closed orbits.
 \end{proposition}
 
 \begin{proof}
 	It is known that the first Betti number of a co-K\"ahler manifold is odd (see \cite[Theorem 11]{chinea1993topology}).
 	By Theorem~\ref{thm:mainthe}, the vector field $R$ is either the Reeb vector field of a Sasakian structure or the Reeb vector field of a co-K\"ahler structure.
 	Since the first Betti number of $M$ is even by assumption, $M$ cannot admit a co-K\"ahler structure.
 	Therefore, $R$ must be the Reeb vector field of a Sasakian structure.
 	Furthermore, by \cite{banyaga1990note, rukimbira1995, rukimbira1994}, the Reeb vector field of a Sasakian manifold admits at least two closed orbits.
 \end{proof} 
 As a direct consequence of Proposition~\ref{thm:remak2}, we deduce that any nowhere-vanishing timelike conformal vector field $R$ on a closed, oriented, smooth $3$-manifold $M$ with finite fundamental group is necessarily the Reeb vector field of a Sasakian structure. Indeed, in this case, the first Betti number of $M$ vanishes. Recall that the first Betti number of a Sasakian manifold is either zero or even (see \cite{Boyer2008}), while for $3$-dimensional co-K\"ahler manifolds, all Betti numbers are odd (see \cite{chinea1993topology}).
 
 \begin{proposition}\label{thm:COCO}
 	Let $R$ be a nowhere-vanishing timelike conformal vector field on a closed, oriented, smooth $3$-manifold $M$. If the first Betti number of $M$ is odd, then $R$ is the Reeb vector field of a co-K\"ahler structure. 
 	
 	Furthermore, there exists a finite covering $\widehat{M}$ of $M$ which is diffeomorphic to the product $\Sigma \times \mathbb{S}^{1}$, where $\Sigma$ is a K\"ahler manifold.
 \end{proposition} 
 \begin{proof}
 	From Proposition \ref{thm:remak2}, $R$ is the Reeb vector field of a co-K\"ahler structure, and by \cite[Theorem 2]{li2008topology}, $M$ is a K\"ahler mapping torus. Then, by \cite[Theorem~3.4]{bazzoni2014}, $M$ is finitely covered by a product $\Sigma \times \mathbb{S}^1$, where $\Sigma$ is a compact K\"ahler surface. This completes the proof.
 \end{proof}
 \begin{remark}
 \emph{	From Proposition~\ref{thm:COCO}, if the first Betti number of $M$ is odd, then $M$ is a co-K\"ahler manifold with co-K\"ahler structure $(\alpha,\xi,\phi,g)$.  
 	Now, following \cite{watson1983}, for each real number $a \in \mathbb{R}$ and $b \ne 0$, the cartesian product $M \times M$ admits a K\"ahler structure $(J_{a,b}, G_{a,b})$ defined by
 \begin{align}
 		J_{a,b}(X_1, X_2) &= 
 		\Bigl(
 		\phi X_1 - \Bigl(\frac{a}{b} \, \alpha(X_1) + \frac{a^2 + b^2}{b} \, \alpha(X_2)\Bigr)\, \xi, \nonumber\\
 		& \phi X_2 + \Bigl(\frac{1}{b} \, \alpha(X_1) + \frac{a}{b} \, \alpha(X_2)\Bigr)\, \xi
 		\Bigr),\nonumber\\ 	 
 		G_{a,b}\bigl((X_1, X_2), (Y_1, Y_2)\bigr) &= 
 		g(X_1, Y_1) + a \, \alpha(X_1) \alpha(Y_2) + a \, \alpha(Y_1) \alpha(X_2) \nonumber\\
 		&  + (a^2 + b^2 + 1) \, \alpha(X_2) \alpha(Y_2) + g(X_2, Y_2).\nonumber
 	\end{align} }
 \end{remark} 
 \begin{remark}
 \emph{In the case of an odd first Betti number, $M$ is a K\"ahler mapping torus ($M = S_{\varphi}$). Let $\mathrm{Isom}(S,h)$ denote the group of Hermitian isometries of $S$, and $\mathrm{Isom}_0(S,h)$ its identity component. By \cite[Theorem~6.6]{bazzoni2014}, the diffeomorphism $\varphi$ is either isotopic to the identity or of finite order in $\mathrm{Isom}(S,h)/\mathrm{Isom}_0(S,h)$. In particular, $\varphi$ is periodic. By \cite[Theorem~5.4]{scott1983geometries}, $M$ is a Seifert fibered space over a $2$-dimensional orbifold $B$ with Euler number zero. Consequently, $M$ admits one of the following Thurston geometries:  	
 	\begin{enumerate}
 		\item $\mathbb{S}^2 \times \mathbb{S}^1$,
 		\item $ \mathbb{H}^2 \times \mathbb{R}/\Gamma$ with $\Gamma \subset \mathrm{Isom}_0(\mathbb{H}^2 \times \mathbb{R})$,
 		\item $\mathbb{T}^2$-bundles over $\mathbb{S}^1$ with periodic monodromy.
 	\end{enumerate}
 }
 \end{remark}

  	\section*{Acknowledgments}	
  
  The authors are grateful to A. Banyaga (PennState University, USA) for invaluable discussions and support.

 	\end{document}